\documentclass[10pt,a4paper]{article}
\usepackage{amsmath,amsthm,amssymb}
\usepackage{geometry}

\usepackage{color}
\usepackage[dvipsnames]{xcolor}
\definecolor{darkred}{rgb}{0.8, 0.0, 0.0}

\usepackage{kotex}

\colorlet{Color0}{Dandelion}
\colorlet{Color1}{red}
\colorlet{Color2}{Green}
\colorlet{Color3}{BlueViolet}

\usepackage{bbm}
\usepackage{url}
\usepackage[normalem]{ulem}
\usepackage[T1]{fontenc}
\usepackage{caption}
\usepackage{subcaption}
\usepackage{hyperref}
\hypersetup{
    colorlinks=true,
    urlcolor = {darkred},
    linkcolor = {darkred},
    citecolor = {darkred},
    linkcolor = {darkred},
    linktoc=all
}
\usepackage{amsthm} 
\usepackage{cleveref}
\usepackage{amssymb}
\usepackage{framed}
\usepackage{graphicx}
\usepackage{tikz}
\usepackage{tikz-cd}
\usetikzlibrary{positioning}
\usetikzlibrary{math}
\usepackage{caption}
\usepackage{subcaption}
\usepackage[utf8]{inputenc}
\usepackage[inline]{enumitem} 
\usepackage{mathtools}
\usepackage{etoolbox}
\usepackage{extpfeil}
\usepackage{soul}
\usepackage{authblk}
\usepackage{mathtools}

\usepackage{multicol}

\usetikzlibrary{matrix,arrows,decorations.pathmorphing,decorations.pathreplacing,patterns,shapes.geometric,fadings}

\theoremstyle{definition}
\newtheorem{theorem}{Theorem}[section]

\newtheorem{proposition}[theorem]{Proposition}

\newtheorem{remark}[theorem]{Remark}

\newtheorem{definition}[theorem]{Definition}
\newtheorem{lemma}[theorem]{Lemma}

\newtheorem*{question}{Question}
\newtheorem{example}[theorem]{Example}
\counterwithin*{claim}{theorem}

\newcommand{\Zplus}{\mathbb{Z}_{\geq 0}}

\newcommand{\SP}{\mathrm{SP}}

\newcommand{\Z}{\mathbb{Z}}
\newcommand{\R}{\mathbb{R}}

\newcommand{\N}{\mathbb{N}}

\newcommand{\RNum}[1]{\uppercase\expandafter{\romannumeral #1\relax}}

\definecolor{darkblue}{rgb}{0.0, 0.0, 0.8}
\definecolor{darkred}{rgb}{0.8, 0.0, 0.0}
\definecolor{darkgreen}{rgb}{0.0, 0.8, 0.0}

\Crefformat{claim}{Claim #2#1#3}
\title{Metric Topologies on Multiset Spaces as Topological Monoids and Their Group Completion}
\author[1]{Donghan Kim\thanks{\texttt{patrick6231@kaist.ac.kr}}}
\affil[1]{Department of Mathematical Sciences, KAIST, South Korea}
\date{}

\begin{document}

\maketitle

\begin{abstract}
We construct a multiset space $\N[X]$ over a metric space $X$ that simultaneously enjoys desirable topological properties and admits a natural matching metric $d_{\N[X]}$, making it a metrizable abelian topological monoid whose structure is compatible with the original metric on $X$.  
This framework extends naturally to the free abelian group $\Z[X]$, where a metric $d_{\Z[X]}$ induces a metrizable abelian topological group structure.  
We further identify the metric completion of $\N[X]$, showing that it carries a canonical extension of the matching metric.
\end{abstract}

\section{Introduction}
A \emph{multiset} is a collection of objects in which repetitions are
allowed, each element occurring with a finite multiplicity.  
When the underlying set $X$ is specified, a multiset of $X$ may be described as
a function $m: X \to \Zplus$ with finite support.  

Multisets are of interest to both mathematicians and computer scientists,
arising naturally in algebra, where they underlie constructions such as algebra of meets and joins~\cite{brink1987some}; in combinatorics,
where they appear in counting problems, prime decomposition, and partition
theory~\cite{aigner2012combinatorial,anderson2002combinatorics,stanley2011enumerative};
and in computer science, where they play a fundamental role in the semantics of
databases, query languages, and data analysis
\cite{grumbach1993towards,libkin1994some,libkin1995representation,singh2007overview}.
Beyond these classical domains, multisets also arise in applied areas such as
information retrieval~\cite{miyamoto2000rough} and knowledge representation~\cite{sheremet2024multiset}, highlighting their broad
and versatile significance.

For a fixed set $X$, the collection of multisets on $X$ carries a natural
abelian monoid structure. A multiset can be identified with a finitely
supported function $m: X \to \Zplus$, and addition is defined pointwise by
\((m+n)(x)=m(x)+n(x)\) for $x \in X$.  
The zero multiset is the constant function $0(x)\equiv 0$.  
Equivalently, if multisets are written as finite formal sums
\(\sum a_x\,x\) with $a_x\in\Zplus$, then addition is given by combining
coefficients~\cite{wildberger2003new}.

From a topological perspective, if $X$ is a topological space, one may ask how
to endow its multiset space with a compatible topology.  
To respect the natural abelian monoid structure, it is desirable to choose a
topology in which the addition operation is continuous.  
Two classical constructions are the infinite symmetric product of
Dold--Thom~\cite{dold1958quasifaserungen} and McCord’s classifying
space~\cite{mccord1969classifying}, which, under mild assumptions on $X$
(in particular, when $X$ is compactly generated), equip the collection of
multisets with the structure of an abelian topological monoid.  
These ideas were further developed in homotopy theory~\cite{bandklayder2019dold} and in the framework of
partial monoids~\cite{mostovoy2007partial}.

From a metric or combinatorial standpoint, a natural question arises: if the
underlying space $X$ is metrizable, can the associated multiset space also be
metrizable?  
Indeed, multisets are often compared using matching or earth–mover type
metrics~\cite{chen2018metric}, and recent work has investigated metric
semigroups of multisets in Banach and combinatorial
contexts~\cite{dolishniak2024metric}.  
Unfortunately, even when $X$ is metrizable, the infinite symmetric product is in
general not metrizable (see Example~\ref{example:non_met}).
This leads us to the following question:
\begin{framed}
\begin{question} 
given a metric space $(X,d)$, can one
construct a metric on the associated multiset space that admits an isometric
embedding of $X$ (thus remaining compatible with the original metric), while
also making the addition operation continuous so that the multiset space becomes
an abelian topological monoid? 
\end{question}
\end{framed}
\vspace{2mm}
Our answer is positive.  
In Section~\ref{sec:N[X]_met} we construct such a space of multisets $\N[X]$, which is
metrizable, admits an isometric embedding of $X$, and carries the structure of a
topological abelian monoid.  
Moreover, allowing negative multiplicities leads naturally to the free abelian
group $\Z[X]$, for which we also establish a compatible metric structure.  
Our main results can be summarized as follows:

\begin{itemize}
  \item We construct the multiset space $\N[X]$ associated to a metric space
    $(X,d)$, endowed with a \emph{matching distance} $d_{\N[X]}$.
    This space is metrizable, carries the structure of a topological abelian
    monoid under addition of multisets, and admits an isometric embedding
    of $X$.
  \item We prove that for a pointed metric space $(X,e)$, if the basepoint $e$ is
    isolated, then $\N[X]$ coincides with the infinite symmetric product $\SP(X)$.
    In contrast, if $e$ is not isolated, the space $\SP(X)$ may fails to be first
    countable and hence is not metrizable.
  \item We extend the construction to the free abelian group $\Z[X]$,
    defining a metric $d_{\Z[X]}$ using positive and negative parts of formal
    sums, making
    $\Z[X]$ a metrizable abelian topological group.
  \item We show that the canonical maps
    \[
      X \;\hookrightarrow\; \N[X] \;\hookrightarrow\; \Z[X]
    \]
    are isometric embeddings with respect to the given metric of $X$.
  \item We establish completeness criteria: in general $\N[X]$ is not complete
    even if $X$ is, but we characterize its completion as the space
    $\overline{\N}[X]$, equipped with the
    extended matching metric $d_\ell$, which is complete whenever $X$ is.
\end{itemize}

\section{Preliminaries}\label{sec:prelim}
Before proceeding, we briefly recall some preliminaries.

\subsection{Basic Topological Properties}
A \textbf{pseudometric} on a set $X$ is a function
\[
d \colon X \times X \to \mathbb{R}
\]
satisfying the following properties for all $x,y,z \in X$:
\begin{enumerate}[label=(\roman*)]
    \item \textbf{Nonnegativity:} $d(x,y) \ge 0$,
    \item \textbf{Symmetry:} $d(x,y)=d(y,x)$, and
    \item  \textbf{Triangle inequality:} $d(x,z) \le d(x,y)+d(y,z)$.
\end{enumerate}
A pseudometric $d$ is called a \textbf{metric} if, in addition,
$d(x,y)=0$ holds if and only if $x=y$.

Given $\epsilon > 0$, the set
\[
B_d(x,\epsilon) := \{\, y \in X \mid d(x,y) < \epsilon \,\}
\]
is called the \textbf{$\epsilon$-ball centered at $x$}. We omit the subscript
$d$ if the metric is clear from context. If $d$ is a metric on $X$, the
collection of all $\epsilon$-balls $B_d(x,\epsilon)$, for $x \in X$ and
$\epsilon > 0$, forms a basis for a topology on $X$, called the
\textbf{metric topology} induced by $d$. A topological space $X$ is said to be
\textbf{metrizable} if there exists a metric $d$ inducing its topology.
A \textbf{metric space} is a metrizable space $X$ together with a metric $d$
that induces this topology.

A topological space $X$ is called \textbf{sequentially compact} if every sequence of points of $X$ has a convergent subsequence.

\begin{lemma}\cite[Theorem 28.2]{Munkres2000Topology}
    Let $X$ be metrizable space. Then, the following are equivalent:
    \begin{enumerate}[label=(\roman*)]
        \item $X$ is compact.
        \item $X$ is sequentially compact.
    \end{enumerate}
\end{lemma}

\begin{definition}[Compactly generated space]
A topological space $X$ is called \emph{compactly generated} or a 
\emph{$k$-space} if it satisfies any of the following equivalent conditions:\cite{lawson1974quotients, Munkres2000Topology, willard2012general}
\begin{enumerate}
    \item The topology on $X$ is \emph{coherent} with the family of its compact
    subspaces; that is, a subset $A \subseteq X$ is open (resp.\ closed) in $X$
    if and only if $A \cap K$ is open (resp.\ closed) in $K$ for every compact
    subspace $K \subseteq X$.
    
    \item The topology on $X$ coincides with the \emph{final topology} with
    respect to the family of all continuous maps $f \colon K \to X$
    from compact spaces $K$.
    
    \item $X$ is a quotient space of a topological sum of compact spaces.
    
    \item $X$ is a quotient space of a weakly locally compact space.
\end{enumerate}
\end{definition}

\begin{proposition}\label{prop:metrizable_is_compactly_generated}
    Every metrizable space is compactly generated.
\end{proposition}
\begin{proof}
Let $X$ be a metrizable space with metric $d$. 
The forward implication is immediate: if $A\subseteq X$ is closed and 
$K\subseteq X$ is compact, then $A\cap K$ is closed in $K$ (with the subspace topology).

Now we prove the converse. By \cite[Lemma~21.2]{Munkres2000Topology}, in a
metrizable space a subset $A\subseteq X$ is closed if and only if it contains the limits
of all convergent sequences in $A$. Suppose that $A\cap K$ is closed in $K$
for every compact $K\subseteq X$. Let $x\in\overline{A}$.
Since $X$ is metrizable, there exists a sequence $(a_n)_{n\ge1}\subseteq A$
converging to $x$.

Consider the subset
\[
K \;:=\; \{x\}\,\cup\,\{a_n \mid n\in\mathbb{N}\}\ \subseteq X.
\]
We claim that $K$ is compact. Indeed, let $\mathcal{U}$ be an open cover of
$K$. Choose $U\in\mathcal{U}$ with $x\in U$. Because $a_n\to x$, there exists
$N$ such that $a_n\in U$ for all $n\ge N$. The finitely many points
$a_1,\dots,a_{N-1}$ are covered by finitely many members of $\mathcal{U}$, so
$\mathcal{U}$ admits a finite subcover of $K$. Hence $K$ is compact.

By our assumption, $A\cap K$ is closed in $K$. The sequence $(a_n)$ lies in
$A\cap K$ and converges to $x\in K$, so (again by
\cite[Lemma~21.2]{Munkres2000Topology}) the limit point $x$ must belong to
$A\cap K$. Therefore $x\in A$. Since $x\in\overline{A}$ was arbitrary, we
conclude that $A$ is closed in $X$.

Thus, a subset of $X$ is closed if and only if its intersection with every compact
subspace is closed, which shows that $X$ is compactly generated.

\end{proof}

\subsection{Direct Limit Topology and Infinite Symmetric Product}

Before introducing the infinite symmetric product, we briefly recall the
direct limit topology, which will serve as a basic tool.
Let $\{X_{n}\}_{n \in \N}$ be an increasing sequence of topological spaces,
so that $X_{n} \subseteq X_{n+1}$ for every $n \in \N$.
The \textbf{direct limit topology} (also called the \textbf{final topology})
on the union
\[
X := \bigcup_{n \in \N} X_{n}
\]
is defined to be the finest topology on $X$ making all inclusions
$i_{n} \colon X_{n} \hookrightarrow X$ continuous.

\begin{proposition}\label{prop:dl-open-iff-intersections}
Let $\{X_{n}\}_{n \in \N}$ be as above and set $X := \bigcup_{n \in \N} X_{n}$.
A subset $U \subseteq X$ is open (with respect to the direct limit topology)
if and only if $U \cap X_{n}$ is open in $X_{n}$ for every $n \in \N$.
\end{proposition}

\begin{proof}
Let
\[
\mathcal{T}' \;:=\; \bigl\{\, U \subseteq X \ \big|\ 
U \cap X_{n} \text{ is open in } X_{n} \text{ for all } n \in \N \,\bigr\}.
\]
First, $\mathcal{T}'$ is a topology on $X$: clearly $\varnothing, X \in \mathcal{T}'$;
arbitrary unions are in $\mathcal{T}'$ because
\[
\Bigl(\bigcup_{\alpha} U_{\alpha}\Bigr) \cap X_{n}
= \bigcup_{\alpha} \bigl(U_{\alpha} \cap X_{n}\bigr),
\]
and finite intersections are in $\mathcal{T}'$ because
\[
\bigl(U \cap V\bigr)\cap X_{n}=(U\cap X_{n})\cap(V\cap X_{n}),
\]
which is open in $X_{n}$ if both $U\cap X_{n}$ and $V\cap X_{n}$ are open.

Next, each inclusion $i_{n}\colon X_{n}\hookrightarrow X$ is continuous for
$\mathcal{T}'$, since for any $U\in\mathcal{T}'$ we have
$i_{n}^{-1}(U)=U\cap X_{n}$ open in $X_{n}$ by definition.
Hence $\mathcal{T}'$ is a topology on $X$ that makes all inclusions continuous.

Finally, if $\mathcal{S}$ is any topology on $X$ making all inclusions $i_{n}$ continuous
and $U\in\mathcal{S}$, then $i_{n}^{-1}(U)=U\cap X_{n}$ is open in $X_{n}$ for every $n$,
so $U\in\mathcal{T}'$. Thus $\mathcal{S}\subseteq\mathcal{T}'$ for every such $\mathcal{S}$,
showing that $\mathcal{T}'$ is the finest topology making all inclusions continuous.
By definition of the direct limit topology, $\mathcal{T}'$ coincides with it,
which proves the characterization.
\end{proof}

\begin{remark}
Some authors take the condition “$U$ is open iff $U\cap X_{n}$ is open in $X_{n}$ for all $n$”
as the \emph{definition} of the direct limit topology~\cite{banakh2012direct, glockner2003direct}; Proposition~\ref{prop:dl-open-iff-intersections}
shows this is equivalent to the “finest topology making the inclusions continuous” viewpoint.
\end{remark}

With this characterization in hand, we now turn to the infinite symmetric product,
equipped with the direct limit topology.
We recall the pointed version of the definition of the infinite symmetric
product~\cite{dold1958quasifaserungen}.
Let $(X,e)$ be a pointed topological space with basepoint $e$.
For each $n \in \N$, the symmetric group $S_{n}$ acts on the product space
$X^{n}$ by coordinatewise permutation.
The quotient of this action is called the \textbf{$n$-th symmetric power}
of $X$ and is denoted by $\SP^{n}(X)$.
The orbit of $(x_{1},\dots,x_{n}) \in X^{n}$ under this action is written
as $[x_{1},\dots,x_{n}]$.

If $(X,e)$ is a based space, it is customary to set $\SP^{0}(X) := \{e\}$.
Moreover, $X^{n}$ embeds into $X^{n+1}$ by
\[
(x_{1},\dots,x_{n}) \longmapsto (x_{1},\dots,x_{n},e),
\]
which induces an embedding $\SP^{n}(X) \hookrightarrow \SP^{n+1}(X)$.
The \emph{infinite symmetric product} is then defined as the direct limit
\[
\SP(X) := \varinjlim_{n} \SP^{n}(X).
\]

Equivalently, $\SP(X)$ can be described without category-theoretic language
as the union of the increasing sequence
\[
\SP^{1}(X) \subseteq \SP^{2}(X) \subseteq \cdots,
\]
endowed with the direct limit topology from
Proposition~\ref{prop:dl-open-iff-intersections}.
In particular, a subset $U \subseteq \SP(X)$ is open if and only if
$U \cap \SP^{n}(X)$ is open in $\SP^{n}(X)$ for every $n \in \N$.
We take the basepoint of $\SP(X)$ to be $[e]$, so that $\SP(X)$ becomes a
based space as well.

\section{The Multiset Metric Space {$\N[X]$}}\label{sec:N[X]_met}
Let $(X,d)$ be a metric space with a distinguished basepoint $e \in X$.  
In this section we introduce a \emph{matching distance} on the free abelian
monoid $\SP(X)$ and show that the topology induced by this metric—independently
of the original direct--limit topology—endows $\SP(X)$ with the structure of a
topological monoid (Theorem~\ref{thm:N[X]_top_monoid}). We further prove that
the canonical inclusion $\SP^n(X) \hookrightarrow \N[X]$ is an isometric
embedding (Theorem~\ref{thm:SPn_iso_embed}), where $\N[X]$ denotes the set
$\SP(X)$ equipped with the topology induced by the matching distance
$d_{\N[X]}$. Finally, we identify conditions under which $\SP(X)$ is
homeomorphic to $\N[X]$ (Theorem~\ref{thm:SP(X)_topology_equivalence}).

\begin{definition}[Matching distance on {$\SP(X)$}]\label{def:SP(X)_metric}
For $[x_1,\dots,x_m]$ and $[y_1,\dots,y_n]$ in $\SP(X)$, choose
$N \ge m+n$ and form the padded lists
\[
\mathbf{x}_N := [x_1,\dots,x_m,\underbrace{e,\dots,e}_{N-m\text{ times}}],
\qquad
\mathbf{y}_N := [y_1,\dots,y_n,\underbrace{e,\dots,e}_{N-n\text{ times}}]
\;\in\; \SP^{N}(X),
\]
where $e$ is the basepoint of $X$. Define the \textbf{matching distance} $d_{\N[X]} : \SP(X) \times \SP(X) \longrightarrow \mathbb{R}_{\ge 0}$ by
\[
d_{\N[X]}([x_1,\dots,x_m],[y_1,\dots,y_n])
:= \inf_{N \ge m+n}\;
\min_{\sigma \in S_N}\;
\sum_{i=1}^{N} d\,\bigl(\mathbf{x}_N(i), \mathbf{y}_N(\sigma(i))\bigr),
\]
where $S_N$ is the symmetric group on $N$ letters, $d(\cdot,\cdot)$ is the
metric on $X$, and $\mathbf{x}_N(i)$ (resp.\ $\mathbf{y}_N(i)$) denotes the
$i$-th entry of the padded list $\mathbf{x}_N$ (resp.\ $\mathbf{y}_N$).
\end{definition}

We now verify that the function $d_{\N[X]}$ is in fact well-defined and that it defines a metric on the set $\SP(X)$.
\begin{proposition}\label{prop:SP(X)_metric}
    The function $d_{\N[X]}$ is well-defined metric on the set $\SP(X)$.
\end{proposition}
\begin{proof}
    \noindent\textbf{(i) well-definedness}
    By induction on the number of occurrences of the basepoint $e$, it suffices to
    show that for any
    \[
    [x_1,\ldots,x_m]=[x_1,\ldots,x_m,e]
    \quad\text{and}\quad
    [y_1,\ldots,y_n]=[y_1,\ldots,y_n,e]
    \quad\text{in }\SP(X),
    \]
    we have
    \[
    d_{\N[X]}([x_1,\ldots,x_m],[y_1,\ldots,y_n])
    =
    d_{\N[X]}([x_1,\ldots,x_m,e],[y_1,\ldots,y_n,e]).
    \]
    First, by the definition of $d_{\N[X]}$ (Definition~\ref{def:SP(X)_metric}), we
    clearly have
    \[
    d_{\N[X]}([x_1,\ldots,x_m],[y_1,\ldots,y_n])
    \;\le\;
    d_{\N[X]}([x_1,\ldots,x_m,e],[y_1,\ldots,y_n,e]).
    \]
    For the reverse inequality, fix $N \ge m+n$ and let
    \[
    \mathbf{x}_N = [x_1,\dots,x_m,\underbrace{e,\dots,e}_{N-m}],\qquad
    \mathbf{y}_N = [y_1,\dots,y_n,\underbrace{e,\dots,e}_{N-n}]
    \]
    be any padded representatives in $\SP^N(X)$. Adjoin one extra copy of $e$ to
    each side and consider the $(N+2)$–lists
    \[
    \mathbf{x}_{N+2} = [\mathbf{x}_N,e,e],\qquad
    \mathbf{y}_{N+2} = [\mathbf{y}_N,e,e].
    \]
    Given any $\sigma \in S_N$, define $\tau \in S_{N+2}$ by
    \[
    \tau(i) := 
    \begin{cases}
    \sigma(i), & 1 \le i \le N,\\
    N+1, & i = N+1,\\
    N+2, & i = N+2.
    \end{cases}
    \]
    Then $\tau \in S_{N+2}$ and extends $\sigma$ on the first $N$ entries.
    Consequently,
    \[
    \sum_{i=1}^{N+2} d\,\bigl(\mathbf{x}_{N+2}(i),\mathbf{y}_{N+2}(\tau(i))\bigr)
    =
    \sum_{i=1}^{N} d\,\bigl(\mathbf{x}_N(i),\mathbf{y}_N(\sigma(i))\bigr)
    + d(e,e) + d(e,e)
    =
    \sum_{i=1}^{N} d\,\bigl(\mathbf{x}_N(i),\mathbf{y}_N(\sigma(i))\bigr).
    \]
    Taking the minimum over $\sigma \in S_N$ and then the infimum over $N \ge m+n$
    gives
    \[
    d_{\N[X]}([x_1,\dots,x_m],[y_1,\dots,y_n])
    \;\ge\;
    d_{\N[X]}([x_1,\dots,x_m,e],[y_1,\dots,y_n,e]),
    \]
    as required. In particular, it follows that
    \[
    d_{\N[X]}([x_1,\dots,x_m],[y_1,\dots,y_n])
    = \min_{\sigma \in S_{N}}\;
    \sum_{i=1}^{N} d\,\bigl(\mathbf{x}_{N}(i), \mathbf{y}_{N}(\sigma(i))\bigr)
    \]
    for every $N \geq m + n$.
    
    \noindent\textbf{(ii) Nonnegativity} 
    For every $[x_1,\dots,x_m]$ and $[y_1,\ldots,y_n]$ in $\SP(X)$ and every $N \geq n + m$, the non-negativity of $d$ implies that 
    \[
    d_{\N[X]}([x_1,\dots,x_m],[y_1,\dots,y_n])
    = \min_{\sigma \in S_{n+m}}\;
    \sum_{i=1}^{n+m} d\,\bigl(\mathbf{x}_{n+m}(i), \mathbf{y}_{n+m}(\sigma(i))\bigr) \geq 0.
    \]
    Now observe that
    \[
    \sum_{i=1}^{m} d(x_i,x_i) + \sum_{i=1}^{m} d(x_i,x_i) = 0,
    \]
    so it follows immediately that
    \[
    d_{\N[X]}([x_1,\dots,x_m],[x_1,\dots,x_m]) = 0.
    \]
    
    Conversely, assume that
    \[
    d_{\N[X]}([x_1,\dots,x_m],[y_1,\dots,y_n]) = 0.
    \]
    Then there exists $\sigma \in S_{n+m}$ such that
    \[
    \sum_{i=1}^{n+m} d\,\bigl(\mathbf{x}_{n+m}(i),\mathbf{y}_{n+m}(\sigma(i))\bigr) = 0.
    \]
    Since each summand is nonnegative, we must have
    \[
    \mathbf{x}_{n+m}(i) = \mathbf{y}_{n+m}(\sigma(i))
    \qquad \text{for all } i=1,\dots,n+m.
    \]
    Hence, after possibly reordering and padding with copies of the basepoint $e$,
    \[
    [x_1,\dots,x_m]
    = [x_1,\dots,x_m,\underbrace{e,\dots,e}_{n \text{ times}}]
    = [y_1,\dots,y_n,\underbrace{e,\dots,e}_{m \text{ times}}]
    = [y_1,\dots,y_n],
    \]
    as required.

    \noindent\textbf{(iii) Symmetry}
        For every $[x_1,\dots,x_m]$ and $[y_1,\ldots,y_n]$ in $\SP(X)$ and every $\sigma \in S_{m+n}$ we have
    \[
    \sum_{i=1}^{m+n} d\,\bigl(\mathbf{x}_{m+n}(i),\mathbf{y}_{m+n}(\sigma(i))\bigr)
    = 
    \sum_{i=1}^{m+n} d\,\bigl(\mathbf{y}_{m+n}(i),\mathbf{x}_{m+n}(\sigma^{-1}(i))\bigr),
    \]
    where $\sigma^{-1}$ denotes the inverse permutation of $\sigma$.
    Taking the minimum over $\sigma \in S_{m+n}$ on both sides shows that
    \[
    d_{\N[X]}([x_1,\dots,x_m],[y_1,\dots,y_n])
    =
    d_{\N[X]}([y_1,\dots,y_n],[x_1,\dots,x_m]),
    \]
    so $d_{\N[X]}$ is symmetric.

    \noindent\textbf{(iv) Triangle inequality}
        Pick arbitrary elements $[x_1,\dots,x_m]$, $[y_1,\ldots,y_n]$, and
    $[z_1,\ldots,z_k]$ of $\SP(X)$ and set $N := m+n+k$.
    Choose $\sigma,\tau \in S_N$ such that
    \[
    d_{\N[X]}([x_1,\ldots,x_m],[y_1,\ldots,y_n])
    = \sum_{i=1}^{N} d\,\bigl(\mathbf{x}_N(i),\mathbf{y}_N(\sigma(i))\bigr)
    \]
    and
    \[
    d_{\N[X]}([y_1,\ldots,y_n],[z_1,\ldots,z_k])
    = \sum_{i=1}^{N} d\,\bigl(\mathbf{y}_N(i),\mathbf{z}_N(\tau(i))\bigr).
    \]
    Consider the composition $\pi := \tau \circ \sigma \in S_N$.
    By the triangle inequality in $X$,
    \[
    d\,\bigl(\mathbf{x}_N(i),\mathbf{z}_N(\pi(i))\bigr)
    \;\le\;
    d\,\bigl(\mathbf{x}_N(i),\mathbf{y}_N(\sigma(i))\bigr)
    + d\,\bigl(\mathbf{y}_N(\sigma(i)),\mathbf{z}_N(\pi(i))\bigr).
    \]
    Summing over $i=1,\dots,N$ and using that $\pi \in S_N$, we obtain
    \begin{align*}
    d_{\N[X]}([x_1,\ldots,x_m],[z_1,\ldots,z_k])
    &\le \sum_{i=1}^{N} d\,\bigl(\mathbf{x}_N(i),\mathbf{z}_N(\pi(i))\bigr) \\
    &\le \sum_{i=1}^{N} d\,\bigl(\mathbf{x}_N(i),\mathbf{y}_N(\sigma(i))\bigr)
    + \sum_{i=1}^{N} d\,\bigl(\mathbf{y}_N(\sigma(i)),\mathbf{z}_N(\pi(i))\bigr) \\
    &= d_{\N[X]}([x_1,\ldots,x_m],[y_1,\ldots,y_n])
    + d_{\N[X]}([y_1,\ldots,y_n],[z_1,\ldots,z_k]).
    \end{align*}
    Since $[x_1,\dots,x_m]$, $[y_1,\dots,y_n]$, and $[z_1,\dots,z_k]$ were
    arbitrary, this shows that $d_{\N[X]}$ satisfies the triangle inequality and is
    therefore a well-defined metric on the set $\SP(X)$.
\end{proof}

From now on, we shall denote the metric space $(\SP(X), d_{\N[X]})$ by
$\N[X]$, and refer to it as the \textbf{multiset space of $X$}.

\begin{remark}\label{remark:d_SP_suff_condition}
    By the proof of Proposition~\ref{prop:SP(X)_metric}, we observe that for every
    $[x_1,\ldots,x_m]$ and $[y_1,\ldots,y_n]$ we have
    \[
    d_{\N[X]}([x_1,\dots,x_m],[y_1,\dots,y_n])
    =
    \min_{\sigma \in S_N}
    \sum_{i=1}^{N} d\,\bigl(\mathbf{x}_N(i),\mathbf{y}_N(\sigma(i))\bigr)
    \qquad\text{for every } N \ge \max(m,n).
    \]
\end{remark}

We now show that $\N[X]$ endows the
structure of an abelian topological monoid.

\begin{theorem}[Abelian Topological Monoid Structure of {$\N[X]$}]
\label{thm:N[X]_top_monoid}
Let $(X,e)$ be a pointed metric space.  
Then the metric $d_{\N[X]}$ introduced in
Definition~\ref{def:SP(X)_metric} induces on $\N[X]$ the structure of an
abelian topological monoid.
\end{theorem}
\begin{proof}
    Define
    \[
    +\;:\; \SP(X) \times \SP(X) \longrightarrow \SP(X)
    \qquad\text{by}\qquad
    [x_1,\ldots,x_m] \times [y_1,\ldots,y_n]
    \;\longmapsto\;
    [x_1,\ldots,x_m,y_1,\ldots,y_n].
    \]
    Since $\SP(X)$ is a metric space, the product $\SP(X) \times \SP(X)$ is
    equipped with the canonical product metric
    \[
    \rho\!\bigl((\mathbf{x},\mathbf{y}),(\mathbf{x}',\mathbf{y}')\bigr)
    \;:=\;
    d_{\N[X]}(\mathbf{x},\mathbf{x}') \;+\; d_{\N[X]}(\mathbf{y},\mathbf{y}').
    \]
    By the definition of $d_{\N[X]}$, taking the infimum over all permutations yields
    \[
    d_{\N[X]}(\mathbf{x},\mathbf{x}')
    \;+\;
    d_{\N[X]}(\mathbf{y},\mathbf{y}')
    \;\geq\;
    d_{\N[X]}(\mathbf{x}+\mathbf{y},\,\mathbf{x}'+\mathbf{y}').
    \]
    Hence the map $+$ is $1$-Lipschitz with respect to $\rho$, and in particular
    it is continuous. 

    It is immediate from the definition that $+$ is associative and commutative.
    Moreover, the singleton multiset $[e]$ acts as a two--sided identity:
    \[
    \mathbf{x} + [e] = \mathbf{x} = [e] + \mathbf{x}
    \qquad
    \text{for all } \mathbf{x} \in \SP(X).
    \]
    Therefore, $\bigl(\SP(X),+\bigr)$ is an abelian topological monoid.
\end{proof}

In contrast to the case of $\SP(X)$, we shall prove that each $\SP^n(X)$ is
metrizable and admits an isometric embedding into $\N[X]$.

\begin{theorem}\label{thm:SPn_iso_embed}
For each $n \in \N$, the space $\SP^{n}(X)$ admits a metric $d_n$ which
induces an isometric embedding into $\N[X]$.
\end{theorem}
Before proving the theorem, we introduce some useful lemmas.

\begin{lemma}[Isometric actions send open balls to open balls]\label{lem:isometric_action_balls}
Let $(M,d)$ be a metric space and let a group $G$ act on $M$ by isometries;
that is, for each $g\in G$ the map $g\colon M\to M$ is bijective and
$d\bigl(gx,gy\bigr)=d(x,y)$ for all $x,y\in M$. Then for every $x\in M$,
$r>0$, and $g\in G$,
\[
g\bigl(B(x,r)\bigr)\;=\;B(gx,r),
\]
where $B(x,r):=\{y\in M: d(x,y)<r\}$.
\end{lemma}

\begin{proof}
Fix $x\in M$, $r>0$, and $g\in G$. 

\smallskip\noindent
($\subseteq$) If $y\in B(x,r)$, then $d(x,y)<r$. By the isometry property,
$d(gx,gy)=d(x,y)<r$, hence $gy\in B(gx,r)$. Thus $g\bigl(B(x,r)\bigr)
\subseteq B(gx,r)$.

\smallskip\noindent
($\supseteq$) Let $z\in B(gx,r)$, so $d(gx,z)<r$. Since $g$ is bijective, there
exists $y:=g^{-1}z\in M$. Using that $g^{-1}$ is also an isometry,
\[
d(x,y)\;=\;d\bigl(g^{-1}gx,g^{-1}z\bigr)\;=\;d(x,g^{-1}z)\;=\;d(gx,z)\;<\;r,
\]
so $y\in B(x,r)$ and hence $z=gy\in g\bigl(B(x,r)\bigr)$. Therefore
$B(gx,r)\subseteq g\bigl(B(x,r)\bigr)$.

Combining the two inclusions yields $g\bigl(B(x,r)\bigr)=B(gx,r)$.
\end{proof}

\begin{lemma}\label{lem:top_equal_SPn}
For each $n \in \N$, the topology on $\SP^{n}(X)$ coincides with the topology
induced by the metric
\[
d_{n} \colon \SP^{n}(X) \times \SP^{n}(X) \longrightarrow \R_{\ge 0},
\qquad
d_{n}([x_1,\ldots,x_n],[y_1,\ldots,y_n])
:= \min_{\sigma \in S_n}
\sum_{i=1}^{n} d(x_i,y_{\sigma(i)}).
\]
\end{lemma}

\begin{proof}[Proof of Lemma~\ref{lem:top_equal_SPn}]
By the same reasoning as in the proof of Proposition~\ref{prop:SP(X)_metric},
one checks that $d_n$ is a well-defined metric on the set $\SP^{n}(X)$:
it is symmetric, nonnegative, vanishes exactly on the diagonal, and satisfies
the triangle inequality.

It remains to show that the topology induced by $d_n$ coincides with the
quotient topology on $\SP^{n}(X)$.

Let $q \colon X^{n} \to \SP^{n}(X)$ be the canonical quotient map. Note that
$X^{n}$ is a metric space with the product metric
\[
d_{X^n}\bigl((x_1,\dots,x_n),(y_1,\dots,y_n)\bigr)
:= \sum_{i=1}^{n} d(x_i,y_i).
\]
By definition of $d_n$, the map
$q$ is $1$-Lipschitz when $\SP^{n}(X)$ is equipped with $d_n$:
\[
d_n\!\bigl(q(\mathbf{x}),q(\mathbf{y})\bigr)
\le d_{\mathrm{X^n}}(\mathbf{x},\mathbf{y}).
\]
Hence $q$ is continuous as a map from $(X^{n},d_{X^n})$ to
$(\SP^{n}(X),d_{n})$. By the universal property of the quotient topology,
this shows that the $d_n$-topology is coarser than the quotient topology.

For the converse inclusion, we first show that $S_n$ acts isometrically on $X^n$.
Indeed, for every $(x_1,\ldots,x_n)$ and $(y_1,\ldots,y_n)$ in $X^n$ and every
$\sigma \in S_n$,
\[
d_{X^n}\bigl((x_1,\ldots,x_n),(y_1,\ldots,y_n)\bigr)
= \sum_{i=1}^{n} d(x_i,y_i)
= \sum_{i=1}^{n} d(x_{\sigma(i)},y_{\sigma(i)})
= d_{X^n}\bigl(\sigma\cdot(x_1,\ldots,x_n),\sigma\cdot(y_1,\ldots,y_n)\bigr).
\]

Let $U \subseteq \SP^{n}(X)$ be any nonempty open set in the quotient topology.
Then $q^{-1}(U)$ is open in $X^{n}$ by definition of the quotient topology.
Pick any $[x_1,\dots,x_n] \in U$ and choose a representative
$\mathbf{x} \in q^{-1}([x_1,\dots,x_n])$. Since $S_n$ is finite, the fiber
$q^{-1}([x_1,\dots,x_n])$ consists of finitely many points
$\{\sigma \cdot \mathbf{x} \mid \sigma \in S_n\}$.
Because $X^n$ is Hausdorff, there exists $\varepsilon>0$ such that
\[
B_{X^n}(\sigma\cdot\mathbf{x},\varepsilon)
\cap
B_{X^n}(\tau\cdot\mathbf{x},\varepsilon)
=\varnothing
\qquad
\text{for all distinct }\sigma,\tau\in S_n,
\]
and, moreover, $B_{X^n}(\mathbf{x},\varepsilon)\subseteq q^{-1}(U)$.
Since $S_n$ acts isometrically on $X^n$, by Lemma~\ref{lem:isometric_action_balls}, we have
\[
\sigma \cdot B_{X^n}(\mathbf{x},\varepsilon)
= B_{X^n}(\sigma\cdot\mathbf{x},\varepsilon)
\qquad\text{for all }\sigma \in S_n.
\]
Hence
\[
q\!\left(\bigcup_{\sigma \in S_n} B_{X^n}(\sigma\cdot\mathbf{x},\varepsilon)\right)
= B_{d_n}\bigl([x_1,\dots,x_n],\varepsilon\bigr)
\subseteq U.
\]

This shows that every $[x_1,\dots,x_n] \in U$ admits a $d_n$–ball contained in
$U$, so $U$ is open in the $d_n$–topology. Therefore, the quotient topology is
contained in the $d_n$–topology.

Combining both inclusions, we conclude that the two topologies coincide.
\end{proof}

Now, we are ready to prove Theorem~\ref{thm:SPn_iso_embed}.
\begin{proof}[Proof of Theorem~\ref{thm:SPn_iso_embed}]
By the definition of $\N[X]$, there is a canonical inclusion of sets
\[
\iota \colon \SP^n(X) \longrightarrow \N[X].
\]
Moreover, by the proof of Lemma~\ref{lem:top_equal_SPn} and
Remark~\ref{remark:d_SP_suff_condition}, the metric $d_n$ on $\SP^n(X)$
coincides with the restriction of $d_{\N[X]}$. Hence, $\iota$ is an
isometric embedding.
\end{proof}

At this point, one may wonder why we work with $\N[X]$ instead of $\SP(X)$.  
The reason is that $\SP(X)$ is not metrizable in general.  
We next investigate sufficient conditions under which $\SP(X)$ is
homeomorphic to $\N[X]$. In addition, we explain why $\SP(X)$ typically fails
to be metrizable with example (Example~\ref{example:non_met}).

\begin{theorem}[Topology Equivalence]\label{thm:SP(X)_topology_equivalence}
Assume that the basepoint $e \in X$ is isolated. Then the metric $d_{\N[X]}$
induces on $\SP(X)$ the same topology as the direct--limit topology.
\end{theorem}
\begin{proof}
    By Theorem~\ref{thm:SPn_iso_embed}, the canonical inclusions
\[
\iota_n \colon (\SP^n(X),d_n)\hookrightarrow (\SP(X),d_{\N[X]})
\]
are isometries for every $n$. Therefore, by the universal property of the colimit, there exists a continuous
map $\N[X] \to \SP(X)$, which shows that the $d_{\N[X]}$–topology is
\emph{coarser} than the direct--limit topology on $\SP(X)$.

It remains to show the reverse inclusion (i.e. the $d_{\N[X]}$–topology is
\emph{finer}). Let $U\subseteq \SP(X)$ be nonempty open in the direct–limit topology
and choose a point $[x_1,\dots,x_n]\in U$ with
$[x_1,\dots,x_n]\notin \SP^{n-1}(X)$. Then $U\cap \SP^n(X)$ is open in
$\SP^n(X)$, so there exists $\varepsilon_0>0$ such that
\[
B_{d_n}\bigl([x_1,\dots,x_n],\varepsilon_0\bigr)\subset U\cap \SP^n(X).
\]
Because $e$ is isolated, there is $r_0>0$ with $d(e,y)\ge r_0$ for all
$y\in X\setminus\{e\}$. Set
\[
\delta \;:=\; \min\{\varepsilon_0,\; r_0/2\}.
\]
We claim that
\[
B_{d_{\N[X]}}\bigl([x_1,\dots,x_n],\delta\bigr)\subset U,
\]
which will complete the proof.

Let $[z]\in B_{d_{\N[X]}}([x_1,\dots,x_n],\delta)$. By definition of $d_{\N[X]}$,
there exists $k\ge n$ such that $[z]=[z_1,\dots,z_k]\in \SP^k(X)$. After
relabeling the entries if necessary, we may assume that
\[
d_k\!\bigl([x_1,\dots,x_n,e,\ldots,e],[z_1,\dots,z_k]\bigr)
= d_n([x_1,\dots,x_n],[z_1,\dots,z_n])
+ d_{k-n}([e],[z_{n+1},\dots,z_k]).
\]
Consequently,
\begin{align*}
d_k\!\bigl([x_1,\dots,x_n,e,\ldots,e],[z_1,\dots,z_k]\bigr)
&= d_n([x_1,\dots,x_n],[z_1,\dots,z_n])
+ d_{k-n}([e],[z_{n+1},\dots,z_k]) \\
&< \delta \;\le\; r_0/2.
\end{align*}
By the choice of $\delta$, the second term must vanish, so
\[
[z_{n+1},\dots,z_k]=[e].
\]
Hence $[z]\in \SP^n(X)$, and therefore
\[
[z]\in B_{d_n}([x_1,\dots,x_n],\delta)
\subseteq U\cap \SP^n(X)\subseteq U.
\]
Thus $B_{d_{\N[X]}}([x_1,\dots,x_n],\delta)\subset U$, showing that $U$ is open
in the $d_{\N[X]}$–topology.

We have proved both inclusions of topologies, so the $d_{\N[X]}$–topology
coincides with the direct–limit topology on $\SP(X)$.

\end{proof}

\begin{remark}
In general, the topology on $\SP(X)$ is always finer than that of $\N[X]$,
since the two constructions share the same underlying set. Moreover, by the
proof of Theorem~\ref{thm:SP(X)_topology_equivalence}, the existence of the
continuous map $\N[X] \to \SP(X)$ does not depend on the assumption that the
basepoint is isolated.
\end{remark}

\begin{remark} 
If $e$ is not isolated, the metric topology induced by $d_{\N[X]}$ can be 
strictly coarser than the direct--limit topology, as demonstrated by the
following example.

\begin{example}\label{example:non_met}
Let $X=[0,1]$ with the usual euclidean metric and take $e=0$ as the basepoint.
For each $n\ge 1$, consider the set
\[
K := \{[1], [\frac{1}{2}], [\frac{1}{3}], \dots\} \subset \SP(X).
\]
Observe that $K \cap \SP^{n}(X)$ is finite for each $n$, hence closed in
$\SP^{n}(X)$. By definition of the direct--limit topology, $K$ is therefore
closed in $\SP(X)$. On the other hand, by the definition of $d_{\N[X]}$ we have
\[
\lim_{n\to\infty} d_{\N[X]}([\frac{1}{n}],[e]) = \lim_{n\to\infty} \frac{1}{n} = 0,
\]
so the sequence $([\frac{1}{n}])_{n\ge 1}$ converges to $[e]$ in the $d_{\N[X]}$--metric.
Since $[e]\notin K$, the closed set $K$ does not contain all its $d_{\N[X]}$--limits,
showing that the $d_{\N[X]}$--topology is strictly coarser than the direct--limit
topology. In particular, $[e]$ does not admit a countable local basis in $\SP(I)$, hence $\SP(I)$ is not metrizable.
\end{example}
\end{remark}

\section{Topological Properties and Completion of ${\N[X]}$}
In this section, we investigate some properties of $\N[X]$ defined in
Section~\ref{sec:N[X]_met}. Proposition~\ref{prop:SPX_connectedness_family} addresses the basic topological
properties of $\N[X]$, while Proposition~\ref{prop:barSP_complete} concerns its
completion.

\begin{proposition}\label{prop:SPX_connectedness_family}
Let $(X,d)$ be a metric space with basepoint $e\in X$, and let $\N[X]$ be
endowed with the matching distance $d_{\N[X]}$ from
Definition~\ref{def:SP(X)_metric}. Then:
\begin{enumerate}[label=(\roman*)]
    \item\label{item:N[X]_top_1} If $X$ is connected, then $\N[X]$ is connected.
    \item\label{item:N[X]_top_2} If $X$ is path connected, then $\N[X]$ is path connected.
\end{enumerate}
\end{proposition}

\begin{proof}
\noindent\textbf{\ref{item:N[X]_top_1}.}
For each $m \geq 1$, consider the subspace $\SP^m(X) \subset \N[X]$.
Since $X$ is connected, so is $X^m$. As the quotient map
$\pi \colon X^m \to \SP^m(X)$ is continuous, it follows that
$\SP^m(X)$ is connected. Moreover, for every $m \geq 1$ the space
$\SP^m(X)$ contains the common point $[e]$. Hence the increasing union
\[
\N[X] \;=\; \bigcup_{m \geq 1} \SP^m(X)
\]
is a union of connected subspaces with a nonempty mutual intersection, and
therefore connected.

\smallskip
\noindent\textbf{\ref{item:N[X]_top_2}.}
For each $m \geq 1$, consider the subspace $\SP^m(X) \subset \N[X]$.
Since $X$ is path connected, so is $X^m$. As the quotient map
$\pi \colon X^m \to \SP^m(X)$ is continuous, it follows that
$\SP^m(X)$ is path connected. Moreover, for every $m \geq 1$ the space
$\SP^m(X)$ contains the common point $[e]$. Hence the increasing union
\[
\N[X] \;=\; \bigcup_{m \geq 1} \SP^m(X)
\]
is a union of path connected subspaces with a nonempty mutual intersection, and
therefore path connected.

\end{proof}

Now, we provide a counterexample showing that $\N[X]$ is not complete in
general (Example~\ref{example:incomplete}), even when $X$ itself is complete, and then describe the form of its
completion (Theorem~\ref{thm:completion_NX}).

\begin{example}[Incompleteness in general]\label{example:incomplete}
Let $X=\mathbb{R}$ with its usual euclidean metric and basepoint $e=0$. It is well known that $X$ is a complete metric space.
Choose a sequence of distinct points $z_k\in X$ with $d(z_k,0)=2^{-k}$,
and set
\[
\mathbf{x}_n := [z_1,\dots,z_n]\in \N[X].
\]
Then for $m>n$,
\[
d_{\N[X]}(\mathbf{x}_m,\mathbf{x}_n)\;\le\;\sum_{k=n+1}^{m} d(z_k,0)
\;=\;\sum_{k=n+1}^{m} 2^{-k}\xrightarrow[n,m\to\infty]{}0,
\]
so $(\mathbf{x}_n)_n$ is Cauchy in $(\SP(X),d_{\N[X]})$.
However, it has no limit in $\N[X]$, since any limit would need to contain
infinitely many points $\{z_k\}$, whereas $\N[X]$ consists of \emph{finite}
multiset. Hence $\N[X]$ is not complete.
\end{example}

For the purpose of constructing the completion, we begin by introducing the
following definition.

\begin{definition}[$\ell^1$ pseudometric on sequences]\label{def:l1_pseudometric}
Let $(X,d)$ be a metric space with basepoint $e\in X$.
Consider the set $X^\ast$ of sequences $\mathbf{x}=(x_i)_{i\in\mathbb{N}}$
such that
\[
\sum_{i=1}^\infty d(x_i,e) \;<\; \infty.
\]
For two such sequences $\mathbf{x}=(x_i)$ and $\mathbf{y}=(y_i)$, define
\[
\tilde d_\ell(\mathbf{x},\mathbf{y})
\;:=\;
\inf_{\sigma \in S_\infty}\;
\sum_{i=1}^\infty d\bigl(x_i,y_{\sigma(i)}\bigr),
\]
where $S_\infty$ is the group of all bijections of $\mathbb{N}$.
Since the summands are nonnegative, the series is well defined in $[0,\infty]$,
and by the triangle inequality one has
\[
\tilde d_\ell(\mathbf{x},\mathbf{y})
\;\leq\;
\sum_{i=1}^\infty d(x_i,e) + \sum_{i=1}^\infty d(e,y_i) < \infty.
\]
Thus $\tilde d_\ell$ is a well-defined function
$X^\ast \times X^\ast \to [0,\infty)$.
\end{definition}

\begin{lemma}\label{lem:l1_pseudometric}
The function $\tilde d_\ell$ is a pseudometric on the set of sequences in $X^\ast$
with finite $\ell^1$-mass.
\end{lemma}

\begin{proof}
Nonnegativity and symmetry are immediate.  
For the triangle inequality, let $\mathbf{x},\mathbf{y},\mathbf{z}$ be such sequences and $\varepsilon>0$.
Choose permutations $\sigma,\tau$ so that
\[
\sum_i d(x_i,y_{\sigma(i)}) \leq \tilde d_\ell(\mathbf{x},\mathbf{y}) + \tfrac{\varepsilon}{2},\quad
\sum_i d(y_i,z_{\tau(i)}) \leq \tilde d_\ell(\mathbf{y},\mathbf{z}) + \tfrac{\varepsilon}{2}.
\]
Then with $\pi := \tau \circ \sigma$,
\[
\sum_i d(x_i,z_{\pi(i)})
\;\leq\;
\sum_i \Bigl( d(x_i,y_{\sigma(i)}) + d(y_{\sigma(i)},z_{\pi(i)}) \Bigr)
\;\leq\; \tilde d_\ell(\mathbf{x},\mathbf{y}) + \tilde d_\ell(\mathbf{y},\mathbf{z}) + \varepsilon.
\]
Since $\varepsilon >0$ is arbitrary, it follows that
\[
\sum_i d(x_i,z_{\pi(i)}) \leq\; \tilde d_\ell(\mathbf{x},\mathbf{y}) + \tilde d_\ell(\mathbf{y},\mathbf{z}).
\]
Taking the infimum over permutations gives
$\tilde d_\ell(\mathbf{x},\mathbf{z}) \leq \tilde d_\ell(\mathbf{x},\mathbf{y}) + \tilde d_\ell(\mathbf{y},\mathbf{z})$.
\end{proof}

\begin{definition}[$\ell^1$-multisets]\label{def:complete_criterion}
Define an equivalence relation $\sim$ on $X^\ast$ by declaring
$\mathbf{x} \sim \mathbf{y}$ if $\tilde d_\ell(\mathbf{x},\mathbf{y})=0$.
Equivalently, $\mathbf{x}$ and $\mathbf{y}$ differ by a bijection of $\mathbb{N}$
(permutation of coordinates) and by inserting or removing finitely many copies of $e$.

The \textbf{$\ell^1$-multiset space} $\overline{\N}[X]$ is the quotient
\[
\overline{\N}[X] \;:=\; X^\ast / \sim.
\]
We denote its elements by $[x_1,x_2,\dots]$ and call them \textbf{$\ell^1$-multisets}.

The pseudometric $\tilde d_\ell$ descends to a genuine metric $d_\ell$ on $\overline{\N}[X]$, called the
\textbf{extended matching distance}.
\end{definition}

\begin{remark}
Finite multisets embed naturally into $\overline{\N}[X]$.
Indeed, there is an inclusion $\iota:\N[X]\hookrightarrow \overline{\N}[X]$ given by
\[
\iota\bigl([x_1,\dots,x_m]\bigr) \;=\; [x_1,\dots,x_m,e,e,\dots].
\]
\end{remark}

\begin{lemma}\label{lem:d1_metric}
The metric $d_\ell$ on $\overline{\N}[X]$ extends $d_{\N[X]}$ in the sense that for all
$\mathbf{x},\mathbf{y}\in \N[X]$,
\[
d_\ell\bigl(\iota(\mathbf{x}),\iota(\mathbf{y})\bigr) \;=\; d_{\N[X]}(\mathbf{x},\mathbf{y}).
\]
Moreover, $\iota:\bigl(\N[X],d_{\N[X]}\bigr)\to\bigl(\overline{\N}[X],d_\ell\bigr)$
is an isometric embedding.
\end{lemma}
\begin{proof}
    It follows directly from the definition of $\overline{\N}[X]$.
\end{proof}

\begin{proposition}[Completeness of {$(\overline{\N}[X],d_\ell)$}]\label{prop:barSP_complete}
    If $(X,d)$ is complete, then $\bigl(\overline{\N}[X],d_\ell\bigr)$ is complete.
\end{proposition}

\begin{proof}
Let $\bigl([\mathbf{x}^{(n)}]\bigr)_{n\ge1}$ be a $d_\ell$--Cauchy sequence in $\overline{\N}[X]$.
Pick representatives $\mathbf{x}^{(n)}=(x^{(n)}_i)_{i\in\N}\in X^\ast$ for each $n$.

Fix $\varepsilon > 0$.
Since $\bigl([\mathbf{x}^{(n)}]\bigr)$ is $d_\ell$--Cauchy, there exists $N > 0$ such that
for all $n \ge N$ we can choose a permutation $\sigma_{n}\in S_\infty$ with
\[
\sum_{i=1}^\infty d\!\bigl(x^{(N)}_i,\,x^{(n)}_{\sigma_{n}(i)}\bigr)
\;<\; \frac{\varepsilon}{2}.
\]
It follows that for every $n,m \geq N$,
\[
\sum_{i=1}^\infty d\!\bigl(x^{(n)}_{\sigma_{n}(i)},\,x^{(m)}_{\sigma_{m}(i)}\bigr)
\;\leq\; 
\sum_{i=1}^\infty d\!\bigl(x^{(N)}_i,\,x^{(n)}_{\sigma_{n}(i)}\bigr) \;+\;
\sum_{i=1}^\infty d\!\bigl(x^{(N)}_i,\,x^{(m)}_{\sigma_{m}(i)}\bigr)
\;<\; \varepsilon.
\]
Thus we may, without loss of generality, replace each $\mathbf{x}^{(n)}$ by a suitably permuted representative.

Since
\[
d\bigl(x^{(n)}_i, x^{(m)}_i \bigr) \;\leq\; d_\ell\bigl(\mathbf{x}^{(n)},\mathbf{x}^{(m)}\bigr),
\]
the sequence $(x^{(n)}_i)_{n}$ is Cauchy in $X$ for every fixed $i$. By the completeness of $X$, there exists $x_i \in X$ with $x^{(n)}_i \to x_i$. Define
$\mathbf{x}=[x_1,x_2,\ldots]$.

Since there exists $n$ with $d_\ell(\mathbf{x}, \mathbf{x}^{(n)}) < \varepsilon$, we have
\[
d_\ell(\mathbf{x}, [e]) \;\leq\; d_\ell(\mathbf{x}, \mathbf{x}^{(n)}) + d_\ell(\mathbf{x}^{(n)},[e]) < \infty,
\]
hence $\mathbf{x} \in \overline{\N}[X]$.

For the fixed $N$ above, choose $K > 0$ such that
\[
\sum_{i \geq K} d(x_i, e) \;<\; \frac{\varepsilon}{2}, 
\qquad
\sum_{i \geq K} d(x^{(N)}_i, e) \;<\; \frac{\varepsilon}{2}.
\]
Then, for every $n \geq N$,
\[
\sum_{i \geq K} d(x^{(n)}_i, e)
\;\leq\;
\sum_{i \geq K} d(x^{(n)}_i, x^{(N)}_i)
\;+\; \sum_{i \geq K} d(x^{(N)}_i, e)
\;\leq\; d_\ell(\mathbf{x}^{(n)},\mathbf{x}^{(N)}) + \frac{\varepsilon}{2} 
\;<\; \varepsilon.
\]

Thus, for every $n \geq N$,
\[
d_\ell\bigl([\mathbf{x}^{(n)}],[\mathbf{x}]\bigr)
\;\leq\;
\sum_{i=1}^{K} d\bigl(x^{(n)}_i,x_i\bigr)
\;+\; \sum_{i>K}\!\!\Bigl[d\bigl(x^{(n)}_i,e\bigr)+d(x_i,e)\Bigr]
\;<\; \varepsilon \;+\; \varepsilon \;+\; \varepsilon
\;=\; 3\varepsilon.
\]
Hence
$d_\ell\bigl([\mathbf{x}^{(n)}],[\mathbf{x}]\bigr)\to 0$ as $n\to\infty$.

\medskip
Therefore every $d_\ell$--Cauchy sequence in $\overline{\N}[X]$ converges, which proves completeness.
\end{proof}

\begin{proposition}\label{prop:finite_dense}
$\N[X]$ is dense in $\overline{\N}[X]$.
\end{proposition}
\begin{proof}
Pick an arbitrary element $\mathbf{x}=[x_1,x_2,\ldots] \in \overline{\N}[X]$ and
fix $\varepsilon > 0$. By the definition of $\overline{\N}[X]$, there exists
$N > 0$ such that
\[
\sum_{i \geq N} d(x_i, e) < \varepsilon.
\]
Define $\mathbf{x}^{(n)}=[x_1,\ldots,x_n] \in \N[X]$. Then, for every $n \geq N$,
\[
d_\ell(\mathbf{x}, \mathbf{x}^{(n)})
= \sum_{i \geq n} d(x_i, e)
\;\leq\; \sum_{i \geq N} d(x_i, e)
< \varepsilon.
\]
In particular,
\[
d_\ell(\mathbf{x}, \mathbf{x}^{(n)}) \;\longrightarrow\; 0
\quad\text{as } n \to \infty.
\]
Thus $\N[X]$ is dense in $\overline{\N}[X]$.
\end{proof}

\begin{theorem}[Metric completion of {$\N[X]$}]\label{thm:completion_NX}
If $(X,d)$ is complete, then the metric completion of $\bigl(\N[X],d_{\N[X]}\bigr)$
is isometric to $\bigl(\overline{\N}[X],d_\ell\bigr)$. In particular, the canonical
embedding $\iota:\N[X]\hookrightarrow \overline{\N}[X]$ has dense image, and
$\overline{\N}[X]$ is complete.
\end{theorem}

\begin{proof}
By Lemma~\ref{lem:d1_metric}, $\iota$ is an isometric embedding.
By Proposition~\ref{prop:barSP_complete}, $\overline{\N}[X]$ is complete,
and by Proposition~\ref{prop:finite_dense} the image of $\N[X]$ is dense.
Hence $\overline{\N}[X]$ realizes the metric completion of $\N[X]$.
\end{proof}

\section{Group Completion {$\Z[X]$}}
In Section~\ref{sec:N[X]_met}, we proved that if $X$ is metrizable, then the
multiset space $\N[X]$ is also metrizable and carries the structure of a
topological abelian monoid (Theorem~\ref{thm:Z[X]_topology_equivalence}). Moreover, the canonical map
$X \hookrightarrow \N[X]$ is an isometric embedding (Theorem~\ref{thm:top_monoid_seq}). In this section, we show
that if $X$ is metrizable, then the free abelian group $\Z[X]$ is likewise
metrizable and admits the structure of a topological abelian group, and that the
composite
\[
X \;\hookrightarrow\; \N[X] \;\hookrightarrow\; \Z[X]
\]
is a sequence of isometric embeddings, with the last map being a homomorphism of topological monoids.

For notational convenience, we represent elements of $\N[X]$ and $\Z[X]$ as
finite formal sums. Specifically, we write
\[
\sum_{i=1}^{n} a_i x_i \;\in\; \N[X] \quad\text{with } a_i \in \N,
\qquad
\sum_{j=1}^{m} b_j y_j \;\in\; \Z[X] \quad\text{with } b_j \in \Z,
\]
where the coefficients record the multiplicities and only finitely many terms
are nonzero.

We begin by defining a metric on $\Z[X]$. Using this metric, and following the
approach of Section~\ref{sec:N[X]_met}, we prove that the addition operation is
continuous by establishing its $1$-Lipschitz continuity.
\begin{definition}[Matching distance on {$\Z[X]$} ]\label{def:Z[X]_metric}
For any $x\in \Z[X]$ there exist finite sums $\sum_{i=1}^{n} a_i x_i$ and
$\sum_{j=1}^{m} b_j y_j$ with $a_i,b_j\in \N$ such that
\[
x \;=\; \sum_{i=1}^{n} a_i x_i \;-\; \sum_{j=1}^{m} b_j y_j .
\]
We set the positive and negative parts of $x$ to be
\[
x^{+} \;:=\; \sum_{i=1}^{n} a_i x_i,
\qquad
x^{-} \;:=\; \sum_{j=1}^{m} b_j y_j .
\]
Given $x,y\in \Z[X]$, define
\[
d_{\Z[X]} \colon \Z[X]\times \Z[X] \longrightarrow \R_{\ge 0},
\qquad
d_{\Z[X]}(x,y)\;:=\; d_{\N[X]}\!\bigl(x^{+}+y^{-},\, y^{+}+x^{-}\bigr),
\]
where $d_{\N[X]}$ denotes the matching distance on $\SP(X)$.
\end{definition}

We now verify that the function $d_{\Z[X]}$ is in fact well-defined and that it defines a metric on the set $\Z[X]$.
\begin{proposition}\label{prop:Z[X]_metric}
The function
\[
d_{\Z[X]}(x,y)\;:=\; d_{\N[X]}\!\bigl(x^{+}+y^{-},\, y^{+}+x^{-}\bigr),
\qquad x,y\in \Z[X],
\]
is well defined and defines a metric on $\Z[X]$.
\end{proposition}
\begin{proof}
\noindent\textbf{(i) Well-definedness.}
Suppose $x = x^{+}-x^{-} = \tilde{x}^{+}-\tilde{x}^{-}$ are two
representations of $x$ with $x^{\pm},\tilde{x}^{\pm}\in \SP(X)$.
Then $x^{+}+\tilde{x}^{-} = \tilde{x}^{+}+x^{-}$ as elements of $\SP(X)$.
Hence
\[
d_{\N[X]}(x^{+}+y^{-},\, y^{+}+x^{-})
= d_{\N[X]}(\tilde{x}^{+}+y^{-},\, y^{+}+\tilde{x}^{-}),
\]
so $d_{\Z[X]}$ is independent of the chosen decomposition of $x$ (and similarly
for $y$). Thus $d_{\Z[X]}$ is well defined.

\smallskip
\noindent\textbf{(ii) Nonnegativity.}
By definition of $d_{\N[X]}$ we have $d_{\Z[X]}(x,y)\ge 0$ for all $x,y\in \Z[X]$.
Moreover, $d_{\Z[X]}(x,y)=0$ if and only if
\[
x^{+}+y^{-} = y^{+}+x^{-}
\qquad \text{in } \SP(X).
\]
This equality is equivalent to $x^{+}-x^{-}=y^{+}-y^{-}$ in $\Z[X]$, i.e.,
$x=y$. Hence $d_{\Z[X]}(x,y)=0$ if and only if $x=y$.

\smallskip
\noindent\textbf{(iii) Symmetry.}
For all $x,y\in \Z[X]$,
\begin{align*}
d_{\Z[X]}(x,y)
&= d_{\N[X]}\!\bigl(x^{+}+y^{-},y^{+}+x^{-}\bigr) \\
&= d_{\N[X]}\!\bigl(y^{+}+x^{-},x^{+}+y^{-}\bigr)
= d_{\Z[X]}(y,x),
\end{align*}
since $d_{\N[X]}$ is symmetric on $\SP(X)$.

\smallskip
\noindent\textbf{(iv) Triangle inequality.}
Let $x,y,z \in \Z[X]$ with disjoint-support decompositions
$x = x^+ - x^-$, $y = y^+ - y^-$, $z = z^+ - z^-$. 
We must show that
\[
d_{\Z[X]}(x,z) \;\le\; d_{\Z[X]}(x,y) + d_{\Z[X]}(y,z).
\]

By definition,
\[
d_{\Z[X]}(x,z) = d_{\N[X]}(x^+ + z^-, z^+ + x^-).
\]
By the triangle inequality for $d_{\N[X]}$ on $\SP(X)$ we have
\begin{align*}
&d_{\N[X]}(x^+ + z^-, z^+ + x^-)=d_{\N[X]}(x^+ + z^- + y^-, z^+ + x^- + y^-)
\\ &\leq 
d_{\N[X]}(x^+ + z^- + y^-, y^+ + x^- + z^-) + d_{\N[X]}(y^+ + x^- + z^-, z^+ + x^- + y^-).
\end{align*}

Because $d_{\N[X]}$ is insensitive to adding the same multiset to both arguments,
we may cancel the common term $z^-$ in the first summand and $x^-$ in the
second summand, yielding
\[
d_{\N[X]}(x^+ + z^- + y^-, y^+ + x^- + z^-) = d_{\N[X]}(x^+ + y^-, y^+ + x^-)
= d_{\Z[X]}(x,y),
\]
and
\[
d_{\N[X]}(y^+ + x^- + z^-, z^+ + x^- + y^-) = d_{\N[X]}(y^+ + z^-, z^+ + y^-)
= d_{\Z[X]}(y,z).
\]

Therefore,
\[
d_{\Z[X]}(x,z) \;\le\; d_{\Z[X]}(x,y) + d_{\Z[X]}(y,z),
\]
as desired.

\smallskip
Since $d_{\Z[X]}$ satisfies nonnegativity, symmetry, and the triangle inequality, it is a metric on $\Z[X]$.
\end{proof}

We now show that the metric topology on $\Z[X]$ indeed endows it with the
structure of an abelian topological group.

\begin{theorem}[Abelian Topological Group Structure of {$\Z[X]$}]
\label{thm:Z[X]_topology_equivalence}
Let $(X,e)$ be a pointed metric space.  
Then the metric $d_{\Z[X]}$ introduced in
Definition~\ref{def:Z[X]_metric} induces on $\Z[X]$ the structure of an abelian topological group.
\end{theorem}

\begin{proof}

Define $+\;:\; \Z[X] \times \Z[X] \longrightarrow \Z[X]$ by
    \[
    \big(\sum_{i=1}^{m} a_i x_i\big) \times \big(\sum_{j=1}^{n} b_j y_j\big)
    \;\longmapsto\;
    \sum_{i=1}^{m} a_i x_i + \sum_{j=1}^{n} b_j y_j.
    \]
    Since $\Z[X]$ is a metric space, the product $\Z[X] \times \Z[X]$ is
    equipped with the canonical product metric
    \[
    \rho\bigl((\mathbf{x},\mathbf{y}),(\mathbf{x}',\mathbf{y}')\bigr)
    \;:=\;
    d_{\N[X]}(\mathbf{x},\mathbf{x}') \;+\; d_{\N[X]}(\mathbf{y},\mathbf{y}').
    \]
    By the definition of $d_{\Z[X]}$, taking the infimum over all permutations yields
    \[
    d_{\Z[X]}(\mathbf{x},\mathbf{x}')
    \;+\;
    d_{\Z[X]}(\mathbf{y},\mathbf{y}')
    \;\geq\;
    d_{\Z[X]}(\mathbf{x}+\mathbf{y},\,\mathbf{x}'+\mathbf{y}').
    \]
    Hence the map $+$ is $1$-Lipschitz with respect to $\rho$, and in particular
    it is continuous. 

    It is immediate from the definition that $+$ is associative and commutative.
    Moreover, the singleton multiset $[e]$ acts as a two--sided identity:
    \[
    \mathbf{x} + [e] = \mathbf{x} = [e] + \mathbf{x}
    \qquad
    \text{for all } \mathbf{x} \in \SP(X).
    \]
    Therefore, $\bigl(\Z[X],+\bigr)$ is an abelian topological group.

\end{proof}

We now show that the inclusion of $\N[X]$ into $\Z[X]$ is both an isometric
embedding and a homomorphism of topological monoids.

\begin{theorem}\label{thm:top_monoid_seq}
Under the hypotheses of Theorem~\ref{thm:Z[X]_topology_equivalence},
the canonical maps
\[
X \;\hookrightarrow\; \N[X] \;\hookrightarrow\; \Z[X]
\]
form a sequence of isometric embeddings, with the second map also being a
homomorphism of topological monoids.
\end{theorem}
\begin{proof}
The first map $X \hookrightarrow \N[X]$ is the case $n=1$ of
Theorem~\ref{thm:SPn_iso_embed}, and hence is an isometric embedding.

For the second map $\iota \colon \N[X] \hookrightarrow \Z[X]$, note that $\iota$
is simply the inclusion of sets, with $\iota(0)=0$ and
$\iota(x+y)=x+y$. Thus $\iota$ is a homomorphism of topological monoids. It
remains to verify that $\iota$ is an isometric embedding. Since
$x,y \in \N[X]$, we have $x^{-}=y^{-}=0$, and therefore
\[
d_{\Z[X]}(\iota(x),\iota(y))
  = d_{\N[X]}(x^{+}+y^{-},\, y^{+}+x^{-})
  = d_{\N[X]}(x,y),
\]
which proves that $\iota$ is indeed an isometric embedding.
\end{proof}

\section{Quotient Metric Spaces}

In this final section we recall a general construction of quotient metrics. We show that collapsing a closed subset of a
metric space yields a natural quotient metric compatible with the quotient
topology.

\begin{proposition}\label{prop:quotient-metric}
Let $(M,d)$ be a metric space and let $H\subset M$ be a nonempty closed subset.  
Define
\[
d(x,H):=\inf_{h\in H} d(x,h)\qquad (x\in M).
\]
On the quotient set $M/H$ (all points of $H$ identified to a single point), set
\[
\overline d\big([x],[y]\big)\;:=\;\min\bigl\{\,d(x,y),\; d(x,H)+d(y,H)\,\bigr\}\qquad([x],[y]\in M/H).
\]
Then $\overline d$ is a well-defined metric on $M/H$. Moreover, the quotient map
$q:M\to M/H$, $q(x)=[x]$, is $1$-Lipschitz.
\end{proposition}

\begin{lemma}\label{lem:distance-to-H}
For all $x,y\in M$ one has
\[
d(x,H)\;\le\; d(x,y) + d(y,H).
\]
\end{lemma}

\begin{proof}
For any $h\in H$, the triangle inequality gives $d(x,h)\le d(x,y)+d(y,h)$.
Taking $\inf_{h\in H}$ yields $d(x,H)\le d(x,y)+d(y,H)$.
\end{proof}

\begin{proof}[Proof of Proposition~\ref{prop:quotient-metric}]
\noindent\textbf{Well-definedness and symmetry.}
The formula depends only on the classes $[x],[y]$ because $d(\cdot,H)$ is class-invariant, and if $x,x'\in H$ then $[x]=[x']$ and $d(x,H)=d(x',H)=0$. Symmetry is immediate.

\noindent\textbf{Positive definiteness.}
If $\overline d([x],[y])=0$, then either $d(x,y)=0$ (hence $[x]=[y]$), or $d(x,H)+d(y,H)=0$, which implies $d(x,H)=d(y,H)=0$. Since $H$ is closed, this forces $x,y\in H$, so $[x]=[y]=[H]$. Thus $\overline d$ vanishes only on the closed subset $H$.

\noindent\textbf{Triangle inequality.}
Fix $x,y,z\in M$. By Lemma~\ref{lem:distance-to-H},
\[
d(x,H)\le d(x,y)+d(y,H),\qquad d(z,H)\le d(z,y)+d(y,H),
\]
and the triangle inequality gives $d(x,z)\le d(x,y)+d(y,z)$. Hence
\[
\overline d([x],[z])
=\min\{\,d(x,z),\ d(x,H)+d(z,H)\}
\le \min\{\,d(x,y)+d(y,z),\ d(x,H)+d(z,H)\}.
\]

It remains to verify that
\begin{align}\label{eq:key}
\min\{\,d(x,y)+d(y,z),\ d(x,H)+d(z,H)\}
&\le \min\{\,d(x,y),\ d(x,H)+d(y,H)\}  \\
&\quad + \min\{\,d(y,z),\ d(y,H)+d(z,H)\}. \nonumber
\end{align}
A case analysis shows this inequality holds in all four possible choices of minima,
using Lemma~\ref{lem:distance-to-H} as needed. Thus the triangle inequality is satisfied.

\noindent\textbf{1-Lipschitz property.}
For all $x,y\in M$,
\[
\overline d\big(q(x),q(y)\big)=\min\{d(x,y),\,d(x,H)+d(y,H)\}\le d(x,y),
\]
so $q$ is $1$-Lipschitz.

All metric axioms are verified; hence $\overline d$ is a metric on $M/H$.
\end{proof}

\begin{remark}
If one uses the same formula on $M$ itself,
\[
d'(x,y):=\min\{d(x,y),\,d(x,H)+d(y,H)\},
\]
then $d'$ is only a pseudometric, since distinct points in $H$ have distance $0$.
Passing to the quotient $M/H$ restores positive definiteness, provided $H$ is closed.
\end{remark}

\begin{proposition}[Topology equivalence]\label{prop:quot-top=metric-top}
With notation as above, the topology induced by the metric $\overline d$ on $M/H$ coincides with the quotient topology along $q:M\to M/H$, $q(x)=[x]$.
\end{proposition}

\begin{proof}
Let $\mathcal T_{\rm met}$ be the metric topology induced by $\overline d$ and $\mathcal T_{\rm quo}$ the quotient topology, i.e.\ $U\in\mathcal T_{\rm quo}\iff q^{-1}(U)$ is open in $M$.

Since $q:(M,d)\to(M/H,\overline d)$ is $1$-Lipschitz, it is continuous. As $\mathcal T_{\rm quo}$ is the finest topology making $q$ continuous, we must have $\mathcal T_{\rm met}\subset \mathcal T_{\rm quo}$.

Let $U\in\mathcal T_{\rm quo}$, so $W:=q^{-1}(U)\subset M$ is open. We show each $[x]\in U$ admits an $\overline d$-ball contained in $U$.

\noindent\textbf{Case $x\notin H$.}  
Because $W$ is open and $x\in W$, the distance to the complement is positive:
\[
r_1:=d\big(x,M\setminus W\big)>0.
\]
Also $d(x,H)>0$. Set $r:=\min\{\,r_1,\ \tfrac{1}{2}d(x,H)\}>0$. If $[y]\in M/H$ with $\overline d([x],[y])<r$, then either $d(x,y)<r\le r_1$ (so $y\in W$), or $d(x,H)+d(y,H)<r\le \tfrac12 d(x,H)$, which is impossible. Hence $[y]\in U$.

\noindent\textbf{Case $x\in H$.}  
Then $[x]=[H]\in U$ and $H\subset W$. As $W$ is open, $\rho:=d(H,M\setminus W)>0$. If $\overline d([H],[y])<\rho$, then $d(y,H)<\rho$, so $y\in W$. Hence $B_{\overline d}([H],\rho)\subset U$.

In both cases, $U$ is open in $\mathcal T_{\rm met}$. Thus $\mathcal T_{\rm quo}\subset \mathcal T_{\rm met}$.

Combining the two steps yields $\mathcal T_{\rm met}=\mathcal T_{\rm quo}$.
\end{proof}

\paragraph*{Acknowledgements}
The author would like to thank the contributors of the MathOverflow discussion 
\emph{``The metrizability of symmetric products of metric spaces''} 
(\url{https://mathoverflow.net/questions/193988/the-metrizability-of-symmetric-products-of-metric-spaces}) 
for helpful insights and references.

\bibliographystyle{plainurl}
\bibliography{bib}

\end{document}